\documentclass[10pt,oneside]{amsart}
\usepackage[latin9]{inputenc}
\usepackage{geometry}
\usepackage{amsthm}
\usepackage{amstext}
\usepackage{amssymb}

\makeatletter

\numberwithin{equation}{section}
\numberwithin{figure}{section}

\title{aaa}

\usepackage{amsfonts}
\usepackage{enumerate}
\usepackage{amssymb}
\usepackage{amsthm}

\theoremstyle{plain}

\newtheorem{defn}{Definition}
\newtheorem{exmp}{Example}

\newtheorem{prop}{Proposition}

\numberwithin{equation}{section}

\makeatother

\begin{document}
\title{Multi-criteria decision making method  based on similarity \\
measures under single valued neutrosophic refined and interval
neutrosophic refined environments }

\author{Faruk Karaaslan}
\email{fkaraaslan@karatekin.edu.tr}
\address{Department of
Mathematics, Faculty of Sciences, \c{C}ank{\i}r{\i} Karatekin
University, 18100, \c{C}ank{\i}r{\i}, Turkey}

\begin{abstract}
In this paper, we propose three similarity measure methods for
single valued neutrosophic refined sets and interval neutrosophic
refined sets based on Jaccard, Dice and Cosine similarity measures
of single valued neutrosophic sets and interval neutrosophic sets.
Furthermore, we suggest two multi-criteria decision making method
under single valued neutrosophic refined environment and interval
neutrosophic refined environment, and give applications of proposed
multi-criteria decision making methods. Finally we suggested a
consistency analysis method for similarity measures between interval
neutrosophic refined sets and give an application to demonstrate
process of the method.
\end{abstract}


\maketitle

\section{Introduction}
To overcome situations containing  uncertainty and inconsistency
data has been very important matter for researchers that study on
mathematical modeling and decision making which is very important in
some areas such as operations research, social economics, and
management science, etc. From  past to  present many studies on
mathematical modeling have been performed. Some of well-known
approximations are fuzzy set (FS) theory proposed by Zadeh
\cite{zadeh-65}, intuitionistic fuzzy set (IFS) theory introduced by
Atanassov \cite{atan-86} and interval valued intuitionistic fuzzy
set theory suggested by Atanassov and Gargov \cite{atan-89}. A FS is
identified by its membership function, IFS which is a generalization
of the FSs is characterized by membership and nonmembership
functions. Even though these set theories are very successful to
model some decision making problems containing uncertainty and
incomplete information, but they may not suffice to model
indeterminate and inconsistent information encountered in real
world. Therefore, Smarandache \cite{smar-98} introduced the concept
of neutrosophic set which is very useful to model problems
containing indeterminate and inconsistent information based on
neutrosophy which is a branch of philosophy. A neutrosophic set is
characterized by three functions called  truth-membership function
($T(x)$), indeterminacy-membership function ($I(x)$) and  falsity
membership function ($F(x)$). These functions are real standard or
nonstandard subsets of $]^-0, 1^+[$, i.e., $T(x): X \to ]^-0, 1^+[,
I(x): X \to ]^-0, 1^+[,$ and $F(x): X \to ]^-0, 1^+[$. Basic of the
neutrosophic set stands up to the non-standard analysis given by
Abraham Robinson in 1960s \cite{abra-66}. Smarandache \cite{smar-06}
discussed comparisons between neutrosophic set, paraconsistent set
and intuitionistic fuzzy set and he shown that the neutrosophic set
is a generalization of paraconsistent set and intuitionistic fuzzy
set. In some areas such as engineering and real scientific fields,
modeling of problems by using real standard or nonstandard subsets
of $]^-0, 1^+[$ may not be easy sometimes, to overcome this issue
concepts of single valued neutrosophic set (SVN-set) and interval
neutrosophic set (IN-set) were defined by Wang et al. in
\cite{wang-06} and \cite{wang-10}, respectively. Zhang et al.
\cite{zha-14} presented an application of IN-set in multi criteria
decision making problems. Some novel operations on interval
neutrosophic sets were defined by Broumi and Smarandache
\cite{bro-14a}. Bhowmik and Pal \cite{bho-10} defined concept of
intuitionistic neutrosophic set by combining intuitionistic fuzzy
set and neutrosophic set, and gave some set theoretical operations
of the intuitionistic neutrosophic set such as complement, union and
intersection. Ansari et al. \cite{ans-11} gave an application of
neutrosophic set theory to medical AI. Ye \cite{ye-15} proposed
concept of trapezoidal neutrosophic set by combining trapezoidal
fuzzy set with single valued neutrosophic set. He also presented
some operational rules related to this novel sets and proposed score
and accuracy function for trapezoidal neutrosophic numbers.

Set theories mentioned above are based on idea which each element of
a set appear only one time in the set. However, in some situations,
a structure containing repeated elements may be need. For instance,
while search in a dad name-number of children-occupation relational
basis. To model such cases, a structure called bags was defined by
Yager \cite{yag-86}. In 1998, Baowen \cite{bao-88} defined concepts
of fuzzy bags and their operations based on Peizhuang's theory of
set-valued statistics \cite{peiz-84} and Yager's bags theory
\cite{yag-86}. Concept of intuitionistic fuzzy bags (multi set)  and
its operations were defined by Shinoj and Sunil \cite{shi-12}, and
they gave an application in medical diagnosis under intuitionistic
fuzzy multi environment.

In 2013, Smarandache \cite{smar-13} put forward  n-symbol or
numerical valued neutrosophic logic which is as generalization of
n-symbol or numerical valued logic that is most general case of
2-valued Boolean logic, Kleene's and Lukasiewicz' 3-symbol valued
logics and Belnap's 4-symbol valued logic. Although existing set
theoretical approximations are generally successful in order to
model some problems encountered in real world, in some cases they
may not allow modeling of problems. For example, when elements in a
set are evaluated by SVN-values in different times  as
$t_1,t_2,...,t_p$, SVN-set may not be sufficient in order to express
such a case. Therefore, Ye and Ye \cite{ye-14svms} defined concept
of single valued neutrosophic multiset (refined) (SVNR-set) as a
generalization of single valued neutrosophic sets, and gave
operational rules for proposed novel set. In a SVNR-set, each of
truth membership values, indeterminacy membership values and falsity
membership values are expressed sequences called truth membership
sequence, indeterminacy membership sequence and falsity membership
sequence, respectively. SVNR-set allows modeling of problems
containing changing values with respect to times under
SVN-environment. In this regard, SVNR-set is an important tool to
model some problems. Bromi et al. \cite{bro-15nvinr} proposed
concept of n-valued interval neutrosophic set and set theoretical
operations on n-valued interval neutrosophic set (or interval
neutrosophic set) such as union, intersection, addition,
multiplication, scalar multiplication, scalar division,
truth-favorite. Also they developed a multi-criteria  group decision
making method and gave its an application in medical diagnosis.

Similarity measure has an important role many areas such as medical
diagnosis, pattern recognition, clustering analysis, decision making
and so on. There are many studies on similarity measures of
neutrosophic sets and IN-sets. For example, Broumi and Smarandache
\cite{bro-13sim} developed  some similarity measure methods between
two neutrosophic sets based on Hausdorff distances and used these
methods to calculate similarity degree between two neutrosophic
sets. Ye \cite{ye-14a-vecsm} proposed three similarity measure
methods used simplified neutrosophic sets (SN-sets) which is a
subclass of neutrosophic set that is more useful than neutrosophic
set  some applications in engineering and real sciences. He also
applied the these methods to decision making problem under
SN-environment. Ye and Zhang \cite{ye-14b-svnsm} suggested
similarity measure between SVN-sets based on minimum and maximum
operators. They also developed a multi-attribute decision making
method based on weighted similarity measure of SVN-sets, and gave
applications to demonstrate effectiveness of the proposed methods.
Ye \cite{ye-14c-cmdbsim} proposed two similarity measures between
SVN-sets by defining  a generalized distance measure, and presented
a clustering algorithm based on proposed similarity measure. In
2015, Ye \cite{ye-15} pointed out some drawbacks of similarity
measures given in \cite{ye-14a-vecsm} and proposed improved cosine
similarity measures of simplified neutrosophic sets (SN-sets) based
on cosine function. Moreover, he defined weighted cosine similarity
measures of SN-sets and gave an application in medical diagnosis
problem containing SN-information. Ye and Fub \cite{ye-15b} proposed
a similarity measure of SVN-sets based on tangent function and put
forward a medical diagnosis method called multi-period medical
diagnosis method based on suggested similarity measure and weighted
aggregation of multi-period information. They also gave a comparison
tangent similarity measures of SVN-sets with existing similarity
measures of SVN-sets. Furthermore, Ye \cite{ye-15c} introduced a
similarity measure of SVN-sets based on cotangent function and gave
an application in the fault diagnosis of steam turbine, and he gave
comparative analysis between cosine similarity measure  and
cotangent similarity measure in the fault diagnosis of steam
turbine. Majumdar and  Samanta \cite{maju-14} defined  notion of
distance between two SVN-sets and investigated its some properties.
They also put forward the a measure of entropy for a SVN-set.
Aydo\u{g}du \cite{ay-14} introduced a similarity measure between two
SVN-sets and developed an entropy of SVN-sets. Bromi and Smarandache
\cite{bro-15a} extended  similarity measures proposed in
\cite{ye-14c-cmdbsim} to IN-sets. Ye \cite{ye-15d} proposed a
similarity measure between two IN-sets based on Hamming and
Euclidian distances and gave a multi-criteria decision making
method.

Similarity measure on the NR-sets was studied Bromi and Smarandache
\cite{bro-14b}. Bromi and Smarandache extended improved cosine
similarity measure of SVN-sets to NR-sets and gave its an
application in medical diagnosis. Mondal and Pramanik \cite{mon-15a}
introduced cotangent similarity measure of NR-sets and studied on
its properties, and applied  cotangent similarity measure to
educational stream selection. Also they   proposed a similarity
measure method \cite{mon-15b} for NR-sets based on tangent function
and gave an application in multi-attribute decision making. In 2015.
Bromi and Smarandache \cite{bro-15a} presented a new similarity
measure method by extending  the Hausdorff distance to NR-sets, and
gave  an application of proposed method in medical diagnosis.

In this paper, we propose three similarity measure methods for
single valued refined sets (SVNR-sets) and interval neutrosophic
refined sets (INR-sets) by extending Jaccard, Dice and Cosine
similarity measures under SVN-value and IN-value given by Ye in
\cite{ye-14a-vecsm}. Also we give two multi-criteria decision making
methods by defining ideal solutions for best and cost criteria under
SVNR-environment and INR-environment. Furthermore, to determine
which similarity measure under INR-environment is  more appropriate
for considered problems, we give a consistency analysis method based
on developed similarity measure methods. To demonstrate processes of
the similarity measure methods and consistency analysis method, we
present real examples based on criteria and attributes given in
\cite{ye-14a-vecsm}. The rest of the article is organized as
follows. In section 2 some concepts  related to the SVN-sets and
IN-sets  and formulas  Jaccard, Dice and Cosine similarity measures
under SVN-environment are given. In section 3 for SVNR-set and
INR-sets similarity measures methods are developed as an extension
of vector similarity measures between SVN-sets and between IN-sets
given in \cite{ye-14a-vecsm}. In section 4 multicritera decision
making methods are developed under SVNR-environment  and
INR-environment, and given examples related to the developed
methods. In section 5 for similarity measures between two INR-sets,
a consistency analysis method is suggested and an application of
this method is given. In section 6 conclusions of the paper  and
studies that can  be made in future are presented.

\section{Preliminary}\label{ss}
In this section,  concepts of SVN-set, IN-set, SVNR-set  and INR-set
 and  some  set theoretical operations of them are presented required
in subsequent sections.

Throughout the paper, $X$ denotes initial universe, $E$ is a set of
parameters and $I_p=\{1,2,...p\}$ is an index set.

\begin{defn}\cite{wang-10}
Let $ X $ be a nonempty  set (initial universe), with a generic
element in $X$ denoted by $x$. A single-valued neutrosophic set
($SVN$-set) $A$ is characterized by a truth membership function
$t_A(x)$, an indeterminacy membership function $i_A(x) $, and a
falsity membership function $f_A(x)$ such that $t_A(x)$, $i_A(x) $,
$f_A(x) \in [0, 1]$ for all $x\in X$, as follows:
\end{defn}

When X is continuous, a SVN-sets A can be written as follows:

$$
  A=
  \int_X\left\langle t_A(x), i_A(x),
  f_A(x) \right\rangle/ x,  \; \; { \rm for \, all } \;x \; \in X.
$$
If $X$ is crisp set, a $SVN$-set $A$ can be written as follows:
$$
  A=
  \sum_{x}\left\langle t_A(x), i_A(x),
  f_A(x) \right\rangle/ x,  \; \; { \rm for \, all } \;x \; \in X.
$$

Also, finite $SVN$-set $A$ can be presented as follows:\\
  $A= \{
  \langle x_{1},t_A(x_{1}),
  i_A(x_{1}),
  f_A(x_{1})
  \rangle,
  \ldots,
  \langle x_{M},
  t_A(x_{M}),
  i_A(x_{M}),
  f_A(x_{M})
  \rangle
  \}$
  for all $ x \in X$.

Here  $ 0\leq t_A(x) +i_A(x) +f_A(x) \leq 3$ for all $ x \in X. $

Throughout this paper, initial universe will be considered as a
finite and crisp set.

From now on set of all $SVN$-sets over $X$ will be denoted by
$SVN_X$.

\begin{defn}\cite{wang-10} Let $A,B\in SVN_X$. Then,
\begin{enumerate}
    \item $A\subseteq B$
  if and only if
  $t_A(x)\leq t_{B}(x), \;
  i_{A}(x)\geq i_{B}(x), \;
  f_{A}(x)\geq f_{B}(x) $
  for any $x \in X$.
  \item $A= B$
  if and only if
  $A\subseteq B$
  and $B\subseteq A$
  for all $x \in X$.
  \item $A^{c}$=$\{\langle x,f_{A}(x),  1- i_{A}(x),
  t_A(x)\rangle:  x \in X \}$.
  \item $A\cup B$=
  $ \{\langle x, (t_A(x)\vee t_B(x)),
  (i_A(x)\wedge i_B(x)),
  (f_A(x)\wedge f_B(x))\rangle: x \in X \}$
  \item $A\cap B$=
  $\{\langle x, (t_A(x)\wedge t_B(x)),
  (i_A(x)\vee i_B(x)),
  (f_{A}(x)\vee f_{B}(x))\rangle:  x \in X\}$.
\end{enumerate}
\end{defn}
\begin{defn}\cite{ye-14svms} Let $X$ be a nonempty set with generic elements in $X$
denoted by $x$.  A single valued neutrosophic refined  set
(SVNR-set) $\tilde A$ is defined as follows:
\begin{eqnarray*} \tilde A=\Big\{\Big\langle x,
(t^1_{A}(x),t^2_{A}(x),...,t^p_{A}(x)),
(i^1_{A}(x),i^2_{A}(x),...,i^p_{A}(x)),
(f^1_{A}(x),f^2_{A}(x),...,f^p_{A}(x))\Big\rangle: x\in X\Big\}.
\end{eqnarray*}

Here, $t^1_{A},t^2_{A},...,t^p_{A}: X \to [0,1]$,
$i^1_{A},i^2_{A},...,i^p_{A}: X \to [0,1]$ and
$f^1_{A},f^2_{A},...,f^p_{A}:X\to [0,1]$ such that $0\leq
t^i_{A}(x)+i^i_{A}(x)+f^i_{A}(x)\leq 3$ for all $x\in X$ and $i\in
I_p$. $(t^1_{A}(x),t^2_{A}(x),...,t^p_{A}(x))$,
$(i^1_{A}(x),i^2_{A}(x),...,i^p_{A}(x))$ and
$(f^1_{A}(x),f^2_{A}(x),...,f^p_{A}(x))$ are the truth-membership
sequence, indeterminacy-membership sequence and falsity-membership
sequence of the element $x$. These sequences may be in decreasing or
increasing order. Also $p$ is called the dimension of single valued
neutrosophic refined set $\tilde{A}$.
\end{defn}
A  $SVNR$-set $A$ can be represented  as follows:
$$
\tilde A=\Big\{\big \langle x,
t^i_{A}(x),i^i_{A}(x),f^i_{A}(x)\big\rangle: x\in X, i\in I_p
\Big\}.
$$
From now on, set of all single valued neutrosophic refined sets over
$X$ will be  denoted by $SVNR_X$ and considered SVNR-sets will be
accepted as $p$ dimension SVNR-set.

\begin{defn}\cite{ye-14svms} Let $\tilde A, \tilde B\in SVNR_X$. Then,
\begin{enumerate}
\item If $t^i_{A}(x)\leq t^i_{\tilde
B}(x)$, $i^i_{A}(x)\geq i^i_{\tilde B}(x)$, $f^i_{A}(x)\geq
f^i_{\tilde B}(x)$  for all $i\in I_p$  and $ x\in X$, then $\tilde
A$ is said to be SVNR-subset of $\tilde B$ and denoted by $\tilde
A\subseteq \tilde B.$
\item $\tilde A\tilde\subseteq \tilde B$ and $\tilde B\tilde\subseteq \tilde A$ if and only if  $\tilde A=\tilde
B$;
\item The complement of $\tilde A$, denoted by $\tilde A^c$, is define as
follows:
$$
\tilde A=\Big\{\big \langle x,
f^i_{A}(x),1-i^i_{A}(x),t^i_{A}(x)\big\rangle: x\in X, i\in I_p
\Big\}.
$$
\end{enumerate}
\end{defn}

\begin{defn}\cite{del-15} Let $\tilde A \in SVNR_X$. Then,
\begin{enumerate}
\item if $t^i_{A}(x)=0$,
$i^i_{A}(x)=1$ and $f^i_{A}(x)=1$ for all $i\in I_p$ and $x\in X$,
$\tilde A$ is called a null SVNR-set, and denoted by $\tilde \Phi$,
\item if $t^i_{A}(x)=1$,
$i^i_{A}(x)=0$ and $f^i_{A}(x)=0$ for all $i\in I_p$ and $x\in X$,
$\tilde A$ is called  universal SVNR-set,  and denoted by $\tilde
X$.
\end{enumerate}
\end{defn}


\begin{defn}\cite{ye-14svms} Let $\tilde A, \tilde B\in SVNR_X$. Then,
\begin{enumerate}
\item union:
$$
\tilde A\tilde\cup \tilde B=\Big\{\big\langle x, t^i_{A}(x)\vee
t^i_{B}(x), i^i_{A}(x)\wedge i^i_{B}(x),f^i_{A}(x)\wedge
f^i_{B}(x)\big\rangle: x\in X, i\in I_p \Big\},
$$
\item intersection:
$$
\tilde A\tilde\cap \tilde B=\Big\{\big\langle x, t^i_{A}(x)\wedge
t^i_{B}(x), i^i_{A}(x)\vee i^i_{B}(x),f^i_{A}(x)\vee
f^i_{B}(x)\big\rangle: x\in X, i\in I_p \Big\}.
$$
\end{enumerate}
\end{defn}

\begin{exmp} \label{ex:svnr}Consider SVNR-sets $\tilde A, \tilde B$ and $\tilde C$ are given as
follows:\\

\footnotesize $\tilde A=\left\{
\begin{array}{c}
\Big \langle x_1 (.1,.2,.4),(.1,.4,.6),(.0,.3,.3)\Big
\rangle, \\
\Big \langle x_2,(.3,.3,.5),(.2,.3,.7),(.1,.5,.6)\Big
\rangle,\\
\Big \langle x_3,(.2,.4,.8),(.1,.3,.3),(.5,.6,.9)\Big
\rangle\\
\end{array}\right\},\quad
\tilde B=\left\{
\begin{array}{c}
\Big \langle x_1 (.5,.6,.7),(.4,.6,.7),(.3,.3,.4)\Big
\rangle, \\
\Big \langle x_2,(.2,.4,.4),(.2,.5,.8),(.2,.6,.7)\Big
\rangle,\\
\Big \langle x_3,(.1,.6,.6),(.1,.5,.5),(.3,.4,.7)\Big
\rangle\\
\end{array}\right\}$\\
\normalsize
 and
 \footnotesize
$$\tilde C=\left\{
\begin{array}{c}
\Big \langle x_1 (.3,.3,.5),(.4,.5,.6),(.1,.3,.4)\Big
\rangle, \\
\Big \langle x_2,(.0,.1,.3),(.2,.3,.6),(.1,.4,.6)\Big
\rangle,\\
\Big \langle x_3,(.1,.4,.7),(.1,.3,.4),(.3,.3,.5)\Big
\rangle\\
\end{array}\right\}.$$\\
\normalsize
 Then,
\footnotesize
 $\tilde A\tilde\cup \tilde B=\left\{
\begin{array}{c}
\Big \langle x_1 (.5,.6,.7),(.4,.6,.7),(.3,.3,.4)\Big
\rangle, \\
\Big \langle x_2,(.2,.3,.4),(.2,.3,.7),(.1,.5,.6)\Big
\rangle,\\
\Big \langle x_3,(.1,.4,.6),(.1,.3,.3),(.3,.4,.7)\Big
\rangle\\
\end{array}\right\},\quad
\tilde A\tilde\cap \tilde B=\left\{
\begin{array}{c}
\Big \langle x_1 (.1,.2,.4),(.1,.4,.6),(.0,.3,.3)\Big
\rangle, \\
\Big \langle x_2,(.3,.4,.5),(.2,.5,.8),(.2,.6,.7)\Big
\rangle,\\
\Big \langle x_3,(.2,.6,.8),(.1,.5,.7),(.5,.6,.9)\Big
\rangle\\
\end{array}\right\},$\\
\normalsize and $\tilde C\tilde \subseteq \tilde B$.
\end{exmp}


\begin{defn}\cite{wang-06}
   Let $ D[0,1] $  be the set of all closed sub-intervals of the interval $ [0,1] $
 and  $ X $ be an ordinary finite non-empty set. An IN-set $\hat A$ over $X$  is set of quadruple given as
 follows:
   $$
   \hat A=\left\{\left\langle x, {t}_{A}(x), {i}_{A}(x), {f}_{A}(x) \right\rangle|x\in X\right\},
   $$
   where, ${t}_{A}(x)\in D[0,1] $,  ${i}_{A}(x)\in D[0,1] $, and  ${f}_{A}(x)\in D[0,1] $  with the relation
   $$
   0\leq sup\, {t}_{A}(x)+sup\, {i}_{A}(x)+sup\, {f}_{A}(x)\leq 3,\; \textrm{for all}\; x \in X.
   $$
   Here intervals $ {t}_{A}(x)$=$ \left[t_{A}^{L}(x),t_{A}^{U}(x)  \right]\subset[0,1]  $,
   $ {i}_{A}(x)$=$ \left[i_{A}^{L}(x),i_{A}^{U}(x)  \right] \subset[0,1] $,
   $ {f}_{A}(x)$=$ \left[f_{A}^{L}(x),f_{A}^{U}(x)  \right] \subset[0,1] $
    denote, respectively the degree of truth, indeterminacy, and falsity membership of $ x \in X $ in $ \tilde A $; moreover
    $ t_{A}^{L}(x) $= $ inf {t}_{A}(x) $,
    $ t_{A}^{U}(x) $= $ sup {t}_{A}(x) $,
    $ i_{A}^{L}(x) $= $ inf {i}_{A}(x) $,
    $ i_{A}^{U}(x) $= $ sup {i}_{A}(x) $,
    $ f_{A}^{L}(x) $= $ inf {f}_{A}(x) $,
    $ f_{A}^{U}(x) $= $ sup {f}_{A}(x) $  for every $ x\in X $.
    Thus, the interval neutrosophic set $\hat A $ can be  expressed in the following interval
    format:
    $$
    \hat{A}=
    \left\lbrace
    \left\langle x,\left[ t_{A}^{L}(x), t_{A}^{U}(x)  \right]
    \left[ i_{A}^{L}(x),i_{A}^{U}(x)  \right]
    \left[f_{A}^{L}(x),f_{A}^{U}(x)  \right]
    \right\rangle|x\in X
    \right\rbrace
    $$
    where, $ 0\leq sup \,t_{A}^{U}(x)+ sup \,i_{A}^{U}(x)+ sup\, f_{A}^{U}(x)\leq 3 $,
    $ T_{A}^{L}(x)\geq 0 $, $ I_{A}^{L}(x)\geq 0 $ and $ F_{A}^{L}(x)\geq 0 $ for all $ x\in X $.
Henceforth  set of all $IN$-sets over $X$ will be denoted by $IN_X$.
\end{defn}

\begin{defn}\cite{bro-15nvinr} Let $X$ be a nonempty initial universe whose elements are discrete.
 An n-valued interval neutrosophic refined  set (or interval
 neutrosophic refined set) $\ddot{A}$ is defined as follows:
\begin{eqnarray*} \ddot{A}&=&\Big\{\Big\langle x,
([{t^i}^{L}_A(x),{t^i}^{U}_A(x)]),
([{i^i}^{L}_A(x),{i^i}^{U}_A(x)]),
([{f^i}^{L}_A(x),{f^i}^{U}_A(x)])\Big\rangle: x\in X \textrm{ and }
i\in I_p\Big\},
\end{eqnarray*}where
${t^i}^{L}_A(x)\leq {t^i}^{U}_A(x), {i^i}^{L}_A(x)\leq
{i^i}^{U}_A(x) \textrm{ and } {f^i}^{L}_A(x)\leq {f^i}^{U}_A(x)$ for
all $x\in X$.

Here,  $0\leq {t^i}^{L}_A(x)+{i^i}^{L}_A(x)+{f^i}^{L}_A(x)\leq 3$
and  $0\leq {t^i}^{U}_A(x)+{i^i}^{U}_A(x)+{f^i}^{U}_A(x)\leq 3$ for
all $x\in X$ and $i\in I_p$.
$$([{t^1}^{L}_A(x),{t^1}^{U}_A(x)],[{t^2}^{L}_A(x),{t^2}^{U}_A(x)],...,[{t^p}^{L}_A(x),{t^p}^{U}_A(x)]),$$
$$([{i^1}^{L}_A(x),{i^1}^{U}_A(x)],[{i^2}^{L}_A(x),{i^2}^{U}_A(x)],...,[{i^p}^{L}_A(x),{i^p}^{U}_A(x)])$$
and
$$([{f^1}^{L}_A(x),{f^1}^{U}_A(x)],[{f^2}^{L}_A(x),{f^2}^{U}_A(x)],...,[{f^p}^{L}_A(x),{f^p}^{U}_A(x)])$$
are  the truth-membership sequence, indeterminacy-membership
sequence and falsity-membership sequence of the element $x$,
respectively.  $p$ is called the dimension of n-valued neutrosophic
sets  $\ddot{A}.$
\end{defn}
Henceforth, considered INR-sets will  be accepted  $p$ dimension
n-valued interval neutrosophic set,  and  set of all interval
neutrosophic refined sets over $X$ will be denoted by $INR_X$. Also
notion of interval neutrosophic refined  set (INR-set) will be used
instead of notion of n-valued interval neutrosophic  set.
\subsection{Similarity measures of SVN-sets and IN-sets}
Jaccard, Dice, and Cosine similarity measures are given between two
$SVN-$sets and between two $IN-$sets defined in \cite{ye-14a-vecsm}.

\begin{defn}
\label{def:ye-jac}
        Let $A $ and $ B $  be two SVN-sets
        in a universe of discourse $ X $ =$\left\{ x_1,x_2, \ldots,x_n
        \right\}$.
        Then the Jaccard similarity measure between SVN-sets $ A $ and $ B $ in the  vector space is defined as follows:
         \begin{equation}\label{eq:jac-sv}
        (A, B)_J=
        \frac{1}{n}\sum_{i=1}^{n}
        \frac{t_{A}(x_{i})t_{B}(x_{i})+
            i_{A}(x_{i})i_{B}(x_{i})
            +f_{A}(x_{i})f_{B}(x_{i})}
        {\left(
           \begin{array}{c}
             (t_{A}^{2}(x_{i})+i_{A}^{2}(x_{i})+
            f_{A}^{2}(x_{i}))+
      (t_{B}^{2}(x_{i})+i_{B}^{2}(x_{i})+f_{B}^{2}(x_{i})) \\
             -(t_{A}(x_{i})t_{B}(x_{i})+
      i_{A}(x_{i})i_{B}(x_{i})+
      f_{A}(x_{i})f_{B}(x_{i})
           \end{array}
         \right)}
        \end{equation}
        \end{defn}
 \begin{defn}
   Let $A $ and $ B $ be two
     SVN-sets  on a universe $ X $ =$\left\{ x_1,x_2, \ldots,x_n
     \right\}$. Then the Dice similarity measure between SVN-sets $ A $ and $ B $ in the  vector space is defined as follows:
      \begin{equation}\label{eq:dice-sv}
      (A,B)_D=
      \frac{1}{n}\sum_{i=1}^{n}
      \frac{2\left( t_{A}(x_{i})t_{B}(x_{i})+
          i_{A}(x_{i})i_{B}(x_{i})
          +f_{A}(x_{i})f_{B}(x_{i})\right)}
      {\left[
          \left(  t_{A}^{2}(x_{i})+i_{A}^{2}(x_{i})+
          f_{A}^{2}(x_{i})\right)+
          \left(t_{B}^{2}(x_{i})+i_{B}^{2}(x_{i})+f_{B}^{2}(x_{i})\right)
          \right] }.
      \end{equation}
  \end{defn}
\begin{defn}
  Let $A $ and
        $ B $  be two
        SVN-sets  in a universe of discourse $ X $ =$\left\{ x_1,x_2, \ldots,x_n
        \right\}$. Then the cosine similarity measure between SVN-sets $ A $ and $ B $ in the  vector space is defined as follows:
        \begin{equation}\label{eq:cos-sv}
        (A, B)_C=
        \frac{1}{n}\sum_{i=1}^{n}
        \frac{\left( t_{A}(x_{i})t_{B}(x_{i})+
            i_{A}(x_{i})i_{B}(x_{i})
            +f_{A}(x_{i})f_{B}(x_{i})\right) }
        {\left[
            \sqrt{\left(t_{A}^{2}(x_{i})+i_{A}^{2}(x_{i})+
            f_{A}^{2}(x_{i})\right)}.
            \sqrt{\left(t_{B}^{2}(x_{i})+i_{B}^{2}(x_{i})+f_{B}^{2}(x_{i})\right)}
            \right] }.
    \end{equation}
\end{defn}

 In some applications, each element  $x_i\in X$ may have different
        weights. Let $w_1,w_2,...,w_n$ be  the weights of elements $x_1,x_2,...,x_n\in
        X$ such that $w_j\geq 0(\forall j\in I_n)$ and $\sum_{j=1}^nw_j=1$,
        respectively. Then, formulas of  Jaccard, Dice and Cosine
        similarity measures between $A $ and $B$ can be extended to weighted Jaccard, Dice and Cosine similarity
measures are defined as follows:
        defined as follows:

       \begin{equation}\label{eq:wjac-sv}
        W(A, B)_{J}=
        \sum_{i=1}^{n}w_{i}
        \frac{t_{A}(x_{i})t_{B}(x_{i})+
            i_{A}(x_{i})i_{B}(x_{i})
            +F_{A}(x_{i})F_{B}(x_{i})}
        {\left(
            \begin{array}{c}
           (t_{A}^{2}(x_{i})+i_{A}^{2}(x_{i})+
            f_{A}^{2}(x_{i}))+
            (t_{B}^{2}(x_{i})+i_{B}^{2}(x_{i})+f_{B}^{2}(x_{i}))\\
            -(t_{A}(x_{i})t_{B}(x_{i})+
            i_{A}(x_{i})i_{B}(x_{i})+
            f_{A}(x_{i})f_{B}(x_{i}))
            \end{array}
            \right)},
 \end{equation}

      \begin{equation}\label{eq:wdice-sv}
     W(A, B)_D=
     \sum_{i=1}^{n}w_{i}
     \frac{2\left( t_{A}(x_{i})t_{B}(x_{i})+
         t_{A}(x_{i})t_{B}(x_{i})
         +f_{A}(x_{i})f_{B}(x_{i})\right) }
        {\left(
            \left(t_{A}^{2}(x_{i})+i_{A}^{2}(x_{i})+
            f_{A}^{2}(x_{i})\right)+
            \left(t_{B}^{2}(x_{i})+i_{B}^{2}(x_{i})+f_{B}^{2}(x_{i})\right)
            \right] }
      \end{equation}
and
      \begin{equation}\label{eq:wcos-sv}
       W(A, B)_C=
        \sum_{i=1}^{n}w_{i}
        \frac{\left( t_{A}(x_{i})t_{B}(x_{i})+
            i_{A}(x_{i})i_{B}(x_{i})
            +f_{A}(x_{i})f_{B}(x_{i})\right) }
        {\left(
            \sqrt{\left(  t_{A}^{2}(x_{i})+i_{A}^{2}(x_{i})+
            f_{A}^{2}(x_{i})\right)}.
            \sqrt{\left(t_{B}^{2}(x_{i})+i_{B}^{2}(x_{i})+f_{B}^{2}(x_{i})\right)}
            \right] },
    \end{equation}
respectively.

\section{Similarity measures under  SVNR and INR-environments}
\label{sec:svnr} In this section, similarity measures between two
SVNR-sets and between two INR-sets are defined based on similarity
measures between two SVN-sets and  similarity measures between two
IN-sets given in \cite{ye-14a-vecsm}.
\begin{defn}
 Let $\tilde A, \tilde B \in SVNR_X$ .  Then, the Jaccard similarity measure between
SVNR-sets $\tilde  A $ and $ \tilde B $  is defined as follows:
        \begin{equation}\label{eq:jac-svr}
        (\tilde A, \tilde B)_J=
        \frac{1}{n}\sum_{j=1}^{n}\sum_{i=1}^p
        \frac{(t^i_{A}(x_{j})t^i_{B}(x_{j})+
            i^i_{A}(x_j)i^i_{B}(x_{j})
            +f^i_{A}(x_{j})f^i_{B}(x_{j}))}
        {\left(
            \begin{array}{c}([t^i_{A}(x_{j})]^{2}+[i^i_{A}(x_{j})]^{2}+
            [f^i_{A}(x_{j})]^{2})+([t^i_{B}(x_{j})]^{2}+[i^i_{B}(x_{j})]^{2}+
            [f^i_{B}(x_{j})]^{2})\\ -[t^i_{A}(x_{j})t^i_{B}(x_{j})+
            i^i_{A}(x_j)i^i_{B}(x_{j})
            +f^i_{A}(x_{j})f^i_{B}(x_{j})]\end{array}\right)},
        \end{equation}
\end{defn}
\begin{defn}
 Let $\tilde A,\tilde B \in SVNR_X$. Then the Dice similarity measure between SVNR-sets $
\tilde A $ and $ \tilde B $ is defined as follows:
      \begin{equation}\label{eq:dice-svr}
      (\tilde A, \tilde B)_D=
      \frac{1}{n}\sum_{j=1}^{n}\sum_{i=1}^p
      \frac{2(t^i_{A}(x_{j})t^i_{B}(x_{j})+
            i^i_{A}(x_j)i^i_{B}(x_{j})
            +f^i_{A}(x_{j})f^i_{B}(x_{j}))}
      {\Big([t^i_{A}(x_{j})]^{2}+[i^i_{A}(x_{j})]^{2}+
            [f^i_{A}(x_{j})]^{2})+([t^i_{B}(x_{j})]^{2}+[i^i_{B}(x_{j})]^{2}+
            [f^i_{B}(x_{j})]^{2}\Big)}.
      \end{equation}
\end{defn}

\begin{defn}
  Let $\tilde A,\tilde B \in SVNR_X$.  Then, the cosine similarity measure between
SVNR-sets  $ \tilde A $ and $ \tilde B $  is defined as follows:
            \begin{equation}\label{eq:cos-svr}
        (\tilde A, \tilde B)_C=
        \frac{1}{n}\sum_{j=1}^{n}\sum_{i=1}^p
        \frac{(t^i_{A}(x_{j})t^i_{B}(x_{j})+
            i^i_{A}(x_j)i^i_{B}(x_{j})
            +f^i_{A}(x_{j})f^i_{B}(x_{j}))}
       {\left(
           \sqrt{([t^i_{A}(x_{j})]^{2}+[i^i_{A}(x_{j})]^{2}+
            [f^i_{A}(x_{j})]^{2})}.\sqrt{([t^i_{B}(x_{j})]^{2}+[i^i_{B}(x_{j})]^{2}+
            [f^i_{B}(x_{j})]^{2})}\right)}.
\end{equation}
    \end{defn}
  If $ w_{j}\in[0,1] $ be the weight of each element $ x_{j} $ for $ j=1,2,\ldots,n $
      such that $ \sum_{j=1}^{n}w_{j} =1$, then  the weighted Jaccard, Dice and Cosine similarity measures between SVNR-sets $\tilde A $ and $ \tilde B $
        are defined as follows:
\small
 \begin{equation}\label{eq:wjac-svr}
        W(\tilde A, \tilde B)_J=
        \sum_{j=1}^{n}\sum_{i=1}^pw_{j}
        \frac{(t^i_{A}(x_{j})t^i_{B}(x_{j})+
            i^i_{A}(x_j)i^i_{B}(x_{j})
            +f^i_{A}(x_{j})f^i_{B}(x_{j}))}
        {\left(
            \begin{array}{c}([t^i_{A}(x_{j})]^{2}+[i^i_{A}(x_{j})]^{2}+
            [f^i_{A}(x_{j})]^{2})+([t^i_{B}(x_{j})]^{2}+[i^i_{B}(x_{j})]^{2}+
            [f^i_{B}(x_{j})]^{2})\\ -[t^i_{A}(x_{j})t^i_{B}(x_{j})+
            i^i_{A}(x_j)i^i_{B}(x_{j})
            +f^i_{A}(x_{j})f^i_{B}(x_{j})]\end{array}\right)},
\end{equation}
\begin{equation}\label{eq:wdice-svr}
     W(\tilde A,\tilde B)_D=
     \sum_{j=1}^{n}\sum_{i=1}^pw_{j}
     \frac{2(t^i_{A}(x_{j})t^i_{B}(x_{j})+
            i^i_{A}(x_j)i^i_{B}(x_{j})
            +f^i_{A}(x_{j})f^i_{B}(x_{j}))}
        {\left(([t^i_{A}(x_{j})]^{2}+[i^i_{A}(x_{j})]^{2}+
            [f^i_{A}(x_{j})]^{2})+([t^i_{B}(x_{j})]^{2}+[i^i_{B}(x_{j})]^{2}+
            [f^i_{B}(x_{j})]^{2})\right)}
      \end{equation}
\normalsize      and \small
\begin{equation}\label{eq:wcos-svr}
      W(\tilde A, \tilde B)_C=
        \sum_{j=1}^{n}\sum_{i=1}^pw_j
        \frac{(t^i_{A}(x_{j})t^i_{B}(x_{j})+
            i^i_{A}(x_j)i^i_{B}(x_{j})
            +f^i_{A}(x_{j})f^i_{B}(x_{j}))}
       {\left(
           \sqrt{([t^i_{A}(x_{j})]^{2}+[i^i_{A}(x_{j})]^{2}+
            [f^i_{A}(x_{j})]^{2})}.\sqrt{([t^i_{B}(x_{j})]^{2}+[i^i_{B}(x_{j})]^{2}+
            [f^i_{B}(x_{j})]^{2})}\right)},
\end{equation}
\normalsize
 respectively.
\begin{exmp}  Consider SVNR-sets $\tilde A$ and $\tilde B$
given in Example \ref{ex:svnr}. Then, by using
Eqs.(\ref{eq:jac-svr}),(\ref{eq:dice-svr}) and (\ref{eq:cos-svr}),
similarity measures between SVNR-sets $\tilde A$ and $\tilde B$ are
obtained as in Table \ref{tab:results1}.

\begin{table}[htbp]
    \caption{Similarity measure under SVNR-environment}
    \label{tab:results1}
    \centering
    \begin{tabular}{cc}
        \hline Similarity measures & Values \\
        \hline $(\tilde A, \tilde B)_J$ & $ 0.834 $ \\
        \hline $(\tilde A, \tilde B)_D$ & $ 0.908 $ \\
        \hline $(\tilde A, \tilde B)_C$ & $ 0.928 $ \\
        \end{tabular}
\end{table} If weights of the $x_1,x_2$ and $x_3$  are taken as $w_1=.7$
$w_2=.2$ and $w_3=.1$, respectively. Then, by using
Eqs.(\ref{eq:wjac-svr}),(\ref{eq:wdice-svr}) and
(\ref{eq:wcos-svr}), weighted similarity measures are obtained as in
Table \ref{tab:results2}:

\begin{table}[htbp]
\caption{Weighted similarity measures  under SVNR-environment}
    \label{tab:results2}
    \centering
    \begin{tabular}{cc}
        \hline Similarity measure & Values \\
        \hline $W(\tilde A, \tilde B)_J$ & $ 0.786
 $ \\
        \hline $W(\tilde A, \tilde B)_D$ & $ 0.879
 $ \\
        \hline $W(\tilde A, \tilde B)_C$ & $ 0.429
 $ \\
        \end{tabular}
\end{table}
\end{exmp}
\begin{prop}\label{pro:svnr} Let
 $\tilde A$  and
 $\tilde B$ be
two  $SVNR$-sets.  Then, each similarity measure $(\tilde A, \tilde
B)_{\Lambda} (\Lambda=J,D,C)$ satisfies the following properties:
 \begin{enumerate}
  \item
        $ 0\leq (\tilde A, \tilde B)_{\Lambda} \leq 1$
 \item
        $  (\tilde A,\tilde B)_{\Lambda} =(\tilde B, \tilde A)_{\Lambda} $;
 \item
        $(\tilde A, \tilde B)_{\Lambda}=1 $ if $ \tilde B = \tilde A $
      i.e. $ {t^i}_{A}(x_{i})={t^i}_{B}(x_{i}) $,
      $ {i^i}_{A}(x_{i})={i^i}_{B}(x_{i}) $, and
      $ {f^i}_{A}(x_{i})={f^i}_{B}(x_{i}) $ for every $ x_{j}\in  X $ and $i\in I_p$.
 \end{enumerate}
\end{prop}
\begin{proof}
\begin{enumerate}
\item
For $p=1$ Eq. (\ref{eq:jac-svr}), (\ref{eq:dice-svr})  and
(\ref{eq:cos-svr}) are reduce to Eq. (\ref{eq:jac-sv}),
(\ref{eq:dice-sv})  and (\ref{eq:cos-sv}), respectively.  For all
$i\in I_p (p> 1)$ according to inequality $x^2+y^2\geq 2xy$, for any
$x_j\in X$ we know that

 $$\sum_{i=1}^p([t^i_{A}(x_{j})]^{2}+[t^i_{B}(x_{j})]^{2})  \geq
 2\sum_{i=1}^p([t^i_{A}(x_{j})].[t^i_{B}(x_{j})]),$$

  $$\sum_{i=1}^p([i^i_{A}(x_{j})]^{2}+[i^i_{B}(x_{j})^{2}]) \geq  2\sum_{i=1}^p([i^i_{A}(x_{j})].[i^i_{B}(x_{j})])$$
 $$\sum_{i=1}^p([f^i_{A}(x_{j})]^{2}+[f^i_{B}(x_{j})]^{2}) \geq
2\sum_{i=1}^p([f^i_{A}(x_{j})].[f^i_{B}(x_{j})]),$$

and\\ \small
$\sum_{i=1}^p([t^i_{A}(x_{j})]^{2}+[t^i_{B}(x_{j})]^{2}+[i^i_{A}(x_{j})]^{2}+[i^i_{B}(x_{j})]^{2}+[f^i_{A}(x_{j})]^{2}+[f^i_{B}(x_{j})]^{2})\geq
2\sum_{i=1}^p([t^i_{A}(x_{j})].[t^i_{B}(x_{j})]+[i^i_{A}(x_{j})].[i^i_{B}(x_{j})]+[f^i_{A}(x_{j})].[f^i_{B}(x_{j})]).
$ \normalsize
 Thus,

 \begin{equation}\sum_{i=1}^p
         \frac{(t^i_{A}(x_{j})t^i_{B}(x_{j})+
            i^i_{A}(x_j)i^i_{B}(x_{j})
            +f^i_{A}(x_{j})f^i_{B}(x_{j}))}
        {\left(
            \begin{array}{c}([t^i_{A}(x_{j})]^{2}+[i^i_{A}(x_{j})]^{2}+
            [f^i_{A}(x_{j})]^{2})+([t^i_{B}(x_{j})]^{2}+[i^i_{B}(x_{j})]^{2}+
            [f^i_{B}(x_{j})]^{2})\\ -[t^i_{A}(x_{j})t^i_{B}(x_{j})+
            i^i_{A}(x_j)i^i_{B}(x_{j})
            +f^i_{A}(x_{j})f^i_{B}(x_{j})]\end{array}\right)}\leq 1.
        \end{equation}
and for all $x_j\in X$

\begin{equation}\sum_{j=1}^n\sum_{i=1}^p
         \frac{(t^i_{A}(x_{j})t^i_{B}(x_{j})+
            i^i_{A}(x_j)i^i_{B}(x_{j})
            +f^i_{A}(x_{j})f^i_{B}(x_{j}))}
        {\left(
            \begin{array}{c}([t^i_{A}(x_{j})]^{2}+[i^i_{A}(x_{j})]^{2}+
            [f^i_{A}(x_{j})]^{2})+([t^i_{B}(x_{j})]^{2}+[i^i_{B}(x_{j})]^{2}+
            [f^i_{B}(x_{j})]^{2})\\ -[t^i_{A}(x_{j})t^i_{B}(x_{j})+
            i^i_{A}(x_j)i^i_{B}(x_{j})
            +f^i_{A}(x_{j})f^i_{B}(x_{j})]\end{array}\right)}\leq n.
        \end{equation}
Similarly, Eq.(\ref{eq:dice-svr}) and Eq. (\ref{eq:cos-svr}) are
true.
\item The proof is clear.
\item Let $A=B$. Then, $ {t^i}_{A}(x_{j})={t^i}_{B}(x_{j}) $,
      $ {i^i}_{A}(x_{j})={i^i}_{B}(x_{j}) $, and
      $ {f^i}_{A}(x_{j})={f^i}_{B}(x_{j}) $ for all $ x_{j}\in  X $ and $i\in
      I_p$ and
\small
      \begin{eqnarray*}
      (\tilde A, \tilde B)_J&=&\frac{1}{n}\sum_{j=1}^n\sum_{i=1}^p
         \frac{[t^i_{A}(x_{j}]^2+
            [i^i_{A}(x_j)]^2
            +[f^i_{A}(x_{j})]^2}
        {\left(
            2[t^i_{A}(x_{j})]^{2}+2[i^i_{A}(x_{j})]^{2}+
            2[f^i_{A}(x_{j})]^{2}-([t^i_{A}(x_{j})]^2+
            [i^i_{A}(x_j)]^2+[f^i_{A}(x_{j})]^2)\right)}\\
&=&\frac{1}{n}\sum_{j=1}^n\sum_{i=1}^n1=1.
        \end{eqnarray*}
\end{enumerate}
\normalsize For Dice and Cosine similarity measures, the proofs of
can be made with similar way.
\end{proof} Each similarity measure between two
SVNR-sets $\tilde A =\Big\{\big \langle x,
t^i_A(x),i^i_A(x),f^i_A(x)\big\rangle: x\in X, i\in I_p \Big\}$  and
  $\tilde B=\Big\{\big \langle x,
t^i_B(x),i^i_B(x),f^i_B(x)\big\rangle: x\in X, i\in I_p \Big\}$ are
undefined when  ${t^i}_A(x)={i^i}_A(x)={f^i}_A(x)=0$ and
${t^i}_B(x)={i^i}_B(x)={f^i}_B(x)=0$ for all $x\in X$ and $i\in
I_p$.
\begin{prop} Let
 $\tilde A,\tilde B \in SVNR_X$.  Then, each weighted similarity measure $W(\tilde
A, \tilde B)_{\Lambda} (\Lambda=J,D,C)$ satisfies the following
properties:
 \begin{enumerate}
  \item
        $ 0\leq W(\tilde A, \tilde B)_{\Lambda} \leq 1$,
 \item
        $  W(\tilde A,\tilde B)_{\Lambda} =W(\tilde B, \tilde A)_{\Lambda}
        $,
 \item
        $W(\tilde A, \tilde B)_{\Lambda}=1 $ if $ \tilde B = \tilde A $
      i.e. $ {t^i}_{A}(x_{j})={t^i}_{B}(x_{j}) $,
      $ {i^i}_{A}(x_{j})={i^i}_{B}(x_{j}) $, and
      $ {f^i}_{A}(x_{j})={f^i}_{B}(x_{j}) $ for every $ x_{j}\in  X $ and $i\in I_p$.
 \end{enumerate}
\end{prop}
\begin{proof}The proofs can be made similar way to proof of
Proposition \ref{pro:svnr}.
\end{proof}
Note that, if $w_j (j=1,2,...,n)$ values take as $\frac{1}{n}$, Eqs.
(\ref{eq:wjac-svr}), (\ref{eq:wdice-svr}) and (\ref{eq:wcos-svr})
are reduced Eqs. (\ref{eq:wjac-sv}), (\ref{eq:wdice-sv}) and
(\ref{eq:wcos-sv}), respectively.

Now similarity measures between two INR-sets will be defined  as a
extension  of similarity measures between two IN-sets  given in
\cite{ye-14a-vecsm}.

For convenience, $\ddot{A}=\Big\{\Big\langle x,
([{t^i}^{L}_A(x),{t^i}^{U}_A(x)]),
([{i^i}^{L}_A(x),{i^i}^{U}_A(x)]),
([{f^i}^{L}_A(x),{f^i}^{U}_A(x)])\Big\rangle: x\in X \textrm{ and }
i\in I_p\Big\}$ will be meant by $\ddot{A} \in INR_X$
\begin{defn} Let $\ddot{A},\ddot{B}\in INR_X$. Then, the Jaccard similarity
measure between INR-sets $\ddot{A} $ and $ \ddot{B} $ is defined as
follows: \footnotesize
        \begin{equation}\label{eq:jac-inr}
        (\ddot{A}, \ddot{B})_J=
        \frac{1}{n}\sum_{j=1}^{n}\sum_{i=1}^p
        \frac{\left(\begin{array}{c}({t^i}^L_{A}(x_{j}){t^i}^L_{B}(x_{j})+{t^i}^U_{A}(x_{j}){t^i}^U_{B}(x_{j}))\\+
            ({i^i}^L_{A}(x_{j}){i^i}^L_{B}(x_{j})+{i^i}^U_{A}(i_{j}){i^i}^U_{B}(i_{j}))\\
            +({f^i}^L_{A}(x_{j}){f^i}^L_{B}(x_{j})+{f^i}^U_{A}(x_{j}){f^i}^U_{B}(x_{j}))\end{array}\right)}
        {\left(
            \begin{array}{c}([{t^i}^L_{A}(x_{j})]^{2}+[{i^i}^L_{A}(x_{j})]^{2}+
            [{f^i}^L_{A}(x_{j})]^{2})+([{t^i}^U_{A}(x_{j})]^{2}+[{i^i}^U_{A}(x_{j})]^{2}+[{f^i}^U_{A}(x_{j})]^{2})\\
            +([{t^i}^L_{B}(x_{j})]^{2}+[{i^i}^L_{B}(x_{j})]^{2}+[{f^i}^L_{B}(x_{j})]^{2})+([{t^i}^U_{B}(x_{j})]^{2}+[{i^i}^U_{B}(x_{j})]^{2}
            +[{f^i}^U_{B}(x_{j})]^{2})\\
            -\big[{t^i}^L_{A}(x_{j}){t^i}^L_{B}(x_{j})+{i^i}^L_{A}(x_{j}){i^i}^L_{B}(x_{j})+{f^i}^L_{A}(x_{j}){f^i}^L_{B}(x_{j})\big]\\
            -\big[{t^i}^U_{A}(x_{j}){t^i}^U_{B}(x_{j})+{i^i}^U_{A}(x_{j}){i^i}^U_{B}(x_{j})+{f^i}^U_{A}(x_{j}){f^i}^U_{B}(x_{j})\big]\end{array}\right)}.
        \end{equation}
\end{defn}
\normalsize
\begin{defn}Let
 $\ddot{A},\ddot{B}\in INR_X$.  Then, the Dice similarity
measure between INR-sets $\ddot{A} $ and $ \ddot{B} $  is defined as
follows:\footnotesize
        \begin{equation}\label{eq:dice-inr}
        (\ddot{A}, \ddot{B})_D=
        \frac{1}{n}\sum_{j=1}^{n}\sum_{i=1}^p
        \frac{2\left(\begin{array}{c}({t^i}^L_{A}(x_{j}){t^i}^L_{B}(x_{j})+{t^i}^U_{A}(x_{j}){t^i}^U_{B}(x_{j}))\\+
            ({i^i}^L_{A}(x_{j}){i^i}^L_{B}(x_{j})+{i^i}^U_{A}(i_{j}){i^i}^U_{B}(i_{j}))\\
            +({f^i}^L_{A}(x_{j}){f^i}^L_{B}(x_{j})+{f^i}^U_{A}(x_{j}){f^i}^U_{B}(x_{j}))\end{array}\right)}
        {\left(
            \begin{array}{c}([{t^i}^L_{A}(x_{j})]^{2}+[{i^i}^L_{A}(x_{j})]^{2}+
            [{f^i}^L_{A}(x_{j})]^{2})+([{t^i}^U_{A}(x_{j})]^{2}+[{i^i}^U_{A}(x_{j})]^{2}+[{f^i}^U_{A}(x_{j})]^{2})\\
            +([{t^i}^L_{B}(x_{j})]^{2}+[{i^i}^L_{B}(x_{j})]^{2}+[{f^i}^L_{B}(x_{j})]^{2})+([{t^i}^U_{B}(x_{j})]^{2}+[{i^i}^U_{B}(x_{j})]^{2}
            +[{f^i}^U_{B}(x_{j})]^{2})\end{array}\right)}.
        \end{equation}
\end{defn}
\normalsize
\begin{defn}Let $\ddot{A},\ddot{B}\in INR_X$. Then the cosine similarity
measure between  $ \ddot{A} $ and $ \ddot{B} $ is defined as
follows: \footnotesize
        \begin{equation}\label{eq:cos-inr}
        (\ddot{A}, \ddot{B})_C=
        \frac{1}{n}\sum_{j=1}^{n}\sum_{i=1}^p
        \frac{\left(\begin{array}{c}({t^i}^L_{A}(x_{j}){t^i}^L_{B}(x_{j})+{t^i}^U_{A}(x_{j}){t^i}^U_{B}(x_{j}))\\+
            ({i^i}^L_{A}(x_{j}){i^i}^L_{B}(x_{j})+{i^i}^U_{A}(i_{j}){i^i}^U_{B}(i_{j}))\\
            +({f^i}^L_{A}(x_{j}){f^i}^L_{B}(x_{j})+{f^i}^U_{A}(x_{j}){f^i}^U_{B}(x_{j}))\end{array}\right)}
        {\left(
            \begin{array}{c}\sqrt{([{t^i}^L_{A}(x_{j})]^{2}+[{i^i}^L_{A}(x_{j})]^{2}+
            [{f^i}^L_{A}(x_{j})]^{2})+([{t^i}^U_{A}(x_{j})]^{2}+[{i^i}^U_{A}(x_{j})]^{2}+[{f^i}^U_{A}(x_{j})]^{2})}\\
            \sqrt{([{t^i}^L_{B}(x_{j})]^{2}+[{i^i}^L_{B}(x_{j})]^{2}+[{f^i}^L_{B}(x_{j})]^{2})+([{t^i}^U_{B}(x_{j})]^{2}+[{i^i}^U_{B}(x_{j})]^{2}
            +[{f^i}^U_{B}(x_{j})]^{2})}\end{array}\right)}.
        \end{equation}
\end{defn}
\normalsize
 If  $ w_{j}\in[0,1] $ be the weight of each element $ x_{j} $ for $ j=1,2,\ldots,n $ such that $ \sum_{j=1}^{n}w_{j} =1$,
then  the weighted Jaccard, Dice and Cosine  similarity measures
between INR-sets $ \ddot{A} $ and $ \ddot B $  is defined as
follows:\footnotesize
        \begin{equation}\label{eq:wjac-inr}
        W(\ddot{A}, \ddot{B})_J=
       \sum_{j=1}^{n}\sum_{i=1}^pw_j
        \frac{\left(\begin{array}{c}({t^i}^L_{A}(x_{j}){t^i}^L_{B}(x_{j})+{t^i}^U_{A}(x_{j}){t^i}^U_{B}(x_{j}))\\+
            ({i^i}^L_{A}(x_{j}){i^i}^L_{B}(x_{j})+{i^i}^U_{A}(i_{j}){i^i}^U_{B}(i_{j}))\\
            +({f^i}^L_{A}(x_{j}){f^i}^L_{B}(x_{j})+{f^i}^U_{A}(x_{j}){f^i}^U_{B}(x_{j}))\end{array}\right)}
        {\left(
            \begin{array}{c}([{t^i}^L_{A}(x_{j})]^{2}+[{i^i}^L_{A}(x_{j})]^{2}+
            [{f^i}^L_{A}(x_{j})]^{2})+([{t^i}^U_{A}(x_{j})]^{2}+[{i^i}^U_{A}(x_{j})]^{2}+[{f^i}^U_{A}(x_{j})]^{2})\\
            +([{t^i}^L_{B}(x_{j})]^{2}+[{i^i}^L_{B}(x_{j})]^{2}+[{f^i}^L_{B}(x_{j})]^{2})+([{t^i}^U_{B}(x_{j})]^{2}+[{i^i}^U_{B}(x_{j})]^{2}
            +[{f^i}^U_{B}(x_{j})]^{2})\\
            -\big[{t^i}^L_{A}(x_{j}){t^i}^L_{B}(x_{j})+{i^i}^L_{A}(x_{j}){i^i}^L_{B}(x_{j})+{f^i}^L_{A}(x_{j}){f^i}^L_{B}(x_{j})\big]\\
            -\big[{t^i}^U_{A}(x_{j}){t^i}^U_{B}(x_{j})+{i^i}^U_{A}(x_{j}){i^i}^U_{B}(x_{j})+{f^i}^U_{A}(x_{j}){f^i}^U_{B}(x_{j})\big]\end{array}\right)},
        \end{equation}
       \begin{equation}\label{eq:wdice-inr}
        W(\ddot{A}, \ddot{B})_D=
         \sum_{j=1}^{n}\sum_{i=1}^pw_j
        \frac{2\left(\begin{array}{c}({t^i}^L_{A}(x_{j}){t^i}^L_{B}(x_{j})+{t^i}^U_{A}(x_{j}){t^i}^U_{B}(x_{j}))\\+
            ({i^i}^L_{A}(x_{j}){i^i}^L_{B}(x_{j})+{i^i}^U_{A}(i_{j}){i^i}^U_{B}(i_{j}))\\
            +({f^i}^L_{A}(x_{j}){f^i}^L_{B}(x_{j})+{f^i}^U_{A}(x_{j}){f^i}^U_{B}(x_{j}))\end{array}\right)}
        {\left(
            \begin{array}{c}([{t^i}^L_{A}(x_{j})]^{2}+[{i^i}^L_{A}(x_{j})]^{2}+
            [{f^i}^L_{A}(x_{j})]^{2})+([{t^i}^U_{A}(x_{j})]^{2}+[{i^i}^U_{A}(x_{j})]^{2}+[{f^i}^U_{A}(x_{j})]^{2})\\
            +([{t^i}^L_{B}(x_{j})]^{2}+[{i^i}^L_{B}(x_{j})]^{2}+[{f^i}^L_{B}(x_{j})]^{2})+([{t^i}^U_{B}(x_{j})]^{2}+[{i^i}^U_{B}(x_{j})]^{2}
            +[{f^i}^U_{B}(x_{j})]^{2})\end{array}\right)}.
        \end{equation}
\normalsize
        and
\footnotesize
  \begin{equation}\label{eq:wcos-inr}
        W(\ddot{A}, \ddot{B})_C=
        \sum_{j=1}^{n}\sum_{i=1}^pw_j
        \frac{\left(\begin{array}{c}({t^i}^L_{A}(x_{j}){t^i}^L_{B}(x_{j})+{t^i}^U_{A}(x_{j}){t^i}^U_{B}(x_{j}))\\+
            ({i^i}^L_{A}(x_{j}){i^i}^L_{B}(x_{j})+{i^i}^U_{A}(i_{j}){i^i}^U_{B}(i_{j}))\\
            +({f^i}^L_{A}(x_{j}){f^i}^L_{B}(x_{j})+{f^i}^U_{A}(x_{j}){f^i}^U_{B}(x_{j}))\end{array}\right)}
        {\left(
            \begin{array}{c}\sqrt{([{t^i}^L_{A}(x_{j})]^{2}+[{i^i}^L_{A}(x_{j})]^{2}+
            [{f^i}^L_{A}(x_{j})]^{2})+([{t^i}^U_{A}(x_{j})]^{2}+[{i^i}^U_{A}(x_{j})]^{2}+[{f^i}^U_{A}(x_{j})]^{2})}\\
            \sqrt{([{t^i}^L_{B}(x_{j})]^{2}+[{i^i}^L_{B}(x_{j})]^{2}+[{f^i}^L_{B}(x_{j})]^{2})+([{t^i}^U_{B}(x_{j})]^{2}+[{i^i}^U_{B}(x_{j})]^{2}
            +[{f^i}^U_{B}(x_{j})]^{2})}\end{array}\right)},
        \end{equation}
\normalsize
  respectively.
%
\begin{prop} $\ddot{A},\ddot{B}\in INR_X$.  Then, each similarity measure
$(\ddot{A}, \ddot{B})_{\Lambda} (\Lambda=J,D,C)$ satisfies the
following properties:
 \begin{enumerate}
  \item
        $ 0\leq (\ddot{A}, \ddot{B})_{\Lambda} \leq 1$
 \item
        $  (\ddot{A},\ddot{B})_{\Lambda} =(\ddot{B}, \ddot{A})_{\Lambda} $;
 \item
        $(\ddot{A}, \ddot{B})_{\Lambda}=1 $ if $ \ddot{B} = \ddot{A} $
      i.e. $ [{t^i}^L_{A}(x_{j}),{t^i}^U_{A}(x_{j})]=[{t^i}^L_{B}(x_{j}),{t^i}^U_{B}(x_{j})] $,
      $ [{i^i}^L_{A}(x_{j}),{i^i}^U_{A}(x_{j})]=[{i^i}^L_{B}(x_{j}),{i^i}^U_{B}(x_{j})] $, and
      $ [{f^i}^L_{A}(x_{j}),{f^i}^U_{A}(x_{j})]=[{f^i}^L_{B}(x_{j}),{f^i}^U_{B}(x_{j})] $  for all $ x_{j}\in  X $ and $i\in I_p$.
 \end{enumerate}
\end{prop}
\begin{proof}The proofs can be made similar way to proof of
Proposition \ref{pro:svnr}.
\end{proof}
\begin{prop} $\ddot{A},\ddot{B}\in INR_X$.  Then, each weighted similarity
measure $W(\ddot{A}, \ddot{B})_{\Lambda} (\Lambda=J,D,C)$ satisfies
the following properties:
 \begin{enumerate}
  \item
        $ 0\leq W(\ddot{A}, \ddot{B})_{\Lambda} \leq 1$
 \item
        $  W(\ddot{A},\ddot{B})_{\Lambda} =W(\ddot{B}, \ddot{A})_{\Lambda} $;
 \item
        $W(\ddot{A}, \ddot{B})_{\Lambda}=1 $ if $ \tilde B = \tilde A $
      i.e. $ [{t^i}^L_{A}(x_{j}),{t^i}^U_{A}(x_{j})]=[{t^i}^L_{B}(x_{j}),{t^i}^U_{B}(x_{j})] $,
      $ [{i^i}^L_{A}(x_{j}),{i^i}^U_{A}(x_{j})]=[{i^i}^L_{B}(x_{j}),{i^i}^U_{B}(x_{j})] $, and
      $ [{f^i}^L_{A}(x_{j}),{f^i}^U_{A}(x_{j})]=[{f^i}^L_{B}(x_{j}),{f^i}^U_{B}(x_{j})] $ for every $ x_{j}\in  X $ and $i\in I_p$.
 \end{enumerate}
\end{prop}
\begin{proof}The proofs can be made  similar way to proof of
Proposition \ref{pro:svnr}.
\end{proof}
Note that if
$[{t^i}^L_A(x),{t^i}^U_A(x)]=[0,0]$,$[{i^i}^L_A(x),{i^i}^U_A(x)]=[0,0],[{f^i}^L_A(x),{f^i}^U_A(x)]=[0,0]$
and
$[{t^i}^L_B(x),{t^i}^U_B(x)]=[0,0],[{i^i}^L_B(x),{i^i}^U_B(x)]=[0,0],[{f^i}^L_B(x),{f^i}^U_B(x)]=[0,0]$,
each similarity measure between two INR-sets $\ddot{A}$ and
$\ddot{B}$ are undefined.

\section{Similarity measure based multicriteia decision making under SVNR-environment and INR-environment}
In this section, applications of weighted similarity measures in
multicriteia decision making problems under SVNR-environment and
INR-environment are given.

Let us consider a MCDM problem with $k$ alternatives and $r$
criteria. Let $A=\{A_1,A_2,...,A_k\}$ be a set of alternatives and
$C=\{C_1,C_2,...,C_r\}$ be the set of criteria and
$w=\{w_1,w_2,...w_r\}$ be weights of the criteria $C_j(j=1,2,...,r)$
such that $w_j\geq 0 (j=1,2,...,r)$ and $\sum_{i=}^rw_i=1.$

\subsection{Multi-criteria decision making under SVNR-environment}
 Let $\{A_1,A_2,...,A_k\}$ be a set of alternatives and
$\{C_1,C_2,...,C_r\}$ be a set of criterion. Alternatives $A_i
(i=1,2,...,k)$ are characterized by SVNR-values for each
$C_j(i=1,2,...,r)$ as follows:
$$
A_i=\Big\{\langle
C_j,(t^1_{A_i}(C_j),....,t^p_{A_i}(C_j)),(i^1_{A_i}(C_j),....,i^p_{A_i}(C_j)),(f^1_{A_i}(C_j),....,f^p_{A_i}(C_j))\rangle:
C_j\in C \Big\},
$$
for the sake of shortness, $(t^1_{A_i}(C_j),....,t^p_{A_i}(C_j))$,
$(i^1_{A_i}(C_j),....,i^p_{A_i}(C_j))$  and
$(f^1_{A_i}(C_j),....,f^p_{A_i}(C_j))$ are denoted by
$(t^1_{ij},....,t^p_{ij})$, $(i^1_{ij},....,i^p_{ij})$ and
$(f^1_{ij},....,f^p_{ij})$, respectively. Thus, the evaluation of
the alternative $A_i$ with respect to the criteria $C_j$ made by
expert or decision maker can be briefly written as
$\gamma_{ij}=\langle(t^1_{ij},....,t^p_{ij}),(i^1_{ij},....,i^p_{ij}),\\(f^1_{ij},....,f^p_{ij})\rangle
(i=1,2,...,k; j=1,2,...,r)$. Hence, SVNR-decision matrix
$D=[\gamma_{ij}]_{k\times r}$ can be constructed.

In MCDM environment, to characterize the best alternative properly
in the decision set the notion of the ideal point is used.  To
evaluate the criteria,  two type modifiers called benefit criteria
(BC) and cost criteria (CC) are generally used.

In this study, for benefit criteria (BC) and cost criteria (CC)
ideal SVNR-values denoted by $A^*$ are defined as follows:
 \footnotesize
\begin{itemize}
\item $\gamma_j^*=\langle({t^1}^*_{j},....,{t^p}^*_{j}),({i^1}^*_{j},....,{i^p}^*_{j}),({f^1}^*_{j},....,{f^p}^*_{j})\rangle
=\left\langle\begin{array}{c}(max_i(t^1_{ij}),....,max_i(t^p_{ij})),(min_i(i^1_{ij}),....,min_i(i^p_{ij})),
\\(min_i(f^1_{ij}),....,min_i(f^p_{ij}))\end{array}\right\rangle$
\item $\gamma_j^*=\langle({t^1}^*_{j},....,{t^p}^*_{j}),({i^1}^*_{j},....,{i^p}^*_{j}),({f^1}^*_{j},....,{f^p}^*_{j})\rangle
=\left\langle\begin{array}{c}(min_i(t^1_{ij}),....,min_i(t^p_{ij})),(max_i(i^1_{ij}),....,max_i(i^p_{ij})),
\\(max_i(f^1_{ij}),....,max_i(f^p_{ij}))\end{array}\right\rangle,$
\end{itemize}
\normalsize respectively. Here equations are called positive ideal
solution and negative ideal solution, respectively.

\section*{Algorithm}
\begin{itemize}

\item \textbf{Step 1:} \textbf{Determination of BC and CC criteria.}

\item \textbf{Step 2:} \textbf{Determination of ideal SVNR-values $A^*$}

\item \textbf{Step 3:} \textbf{Calculation of weighted similarity measures}\\
In this step, using one of the Eq. (\ref{eq:wjac-svr}), Eq.
(\ref{eq:wdice-svr}) or Eq.(\ref{eq:wcos-svr}) weighted similarity
measures between the ideal alternative $A^*$ and $A_i(i=1,2,...,k)$
are calculated.

\item \textbf{Step 4:} \textbf{Ranking of the alternative}\\
Considering the values obtained using one of the Eq.
(\ref{eq:wjac-svr}), Eq. (\ref{eq:wdice-svr}) or
Eq.(\ref{eq:wcos-svr}), the ranking order of all the alternatives
can be easily determined.
\end{itemize}

\section*{Illustrative example 1}

Let us consider the decision making problem given in \cite{ye-15d}.
We adapt this decision making problem to SVNR-set. There is an
investment company, which wants to invest a sum money in the best
option. There is a panel with four possible alternatives to invest
the money: (1) $A_1$ is a car company; (2) $A_2$ is a food company;
(3) $A_3$ is a computer company; (4) $A_4$ is an arms company. The
investment company must take a decision according to the three
criteria (1) $C_1$ is the risk; (2) $C_2$ is the growth; (3) $C_3$
is an environmental impact,
The weights of criteria $C_1,C_2$ and $C_3$ are given by $w_1=0.35,
w_2=0.25$ and $w_3=0.40$, respectively. The four alternatives are to
evaluated under the criteria by SVNR-values provided by decision
maker. These values are shown in SVNR-decision matrix as follows:
\scriptsize
$$D=[\gamma_{ij}]_{k\times r}=\left(
  \begin{array}{ccc}
    \langle (.1,.2,.4),(.3,.3,.5),(.2,.4,.8)\Big
\rangle  & \Big \langle (.1,.4,.6),(.2,.3,.7),(.1,.3,.3)\Big \rangle
& \Big \langle (.0,.3,.3),(.1,.5,.6),(.5,.6,.9)\Big
\rangle \\
    \Big \langle (.5,.6,.7),(.2,.4,.4),(.1,.6,.6)\Big
\rangle  & \Big \langle (.4,.6,.7),(.2,.5,.8),(.1,.5,.5)\Big \rangle
& \Big \langle (.3,.3,.4),(.2,.6,.7),(.3,.4,.7)\Big
\rangle  \\
    \Big \langle (.3,.3,.5),(.0,.1,.3),(.1,.4,.7)\Big
\rangle  & \Big \langle (.4,.5,.6),(.2,.3,.6),(.1,.3,.4)\Big \rangle
& \Big \langle (.1,.3,.4),(.1,.4,.6),(.3,.3,.5)\Big
\rangle  \\
    \Big \langle (.2,.4,.9),(.1,.5,.6),(.3,.5,1)\Big
\rangle  & \Big \langle (.0,.2,.4),(.1,.5,.7),(.6,.7,.9)\Big \rangle
& \Big \langle (.8,.8,.9),(.3,.4,.4),(.6,.6,.8)\Big
\rangle  \\
  \end{array}
\right)$$ \normalsize
\begin{itemize}
\item \textbf{Step 1:} Let us consider $C_1$ and $C_2$ as  benefit criteria and  $C_3$ as cost
criterion.
\item \textbf{Step 2:} From SVNR-decision matrix, ideal alternative $A^*$ can be obtained
as follows: \footnotesize
$$A^*=\Big\{\big\langle(.5,.6,.9),(.0,.1,.3),(.1,.4,.6)\big\rangle,\big\langle(.4,.6,.7),(.1,.3,.6),(.1,.3,.3)\big\rangle
,\big\langle(.0,.3,.3),(.3,.6,.7),(.6,.6,.9)\big\rangle\Big\}.$$
\normalsize
\item \textbf{Step 3:} By using the Eqs. (\ref{eq:jac-svr}), (\ref{eq:wjac-svr}),
 (\ref{eq:dice-svr}), (\ref{eq:wdice-svr}), (\ref{eq:cos-svr}) and (\ref{eq:wcos-svr}), for $\Lambda\in\{J,D,C\}$, similarity
measures and weighted similarity measures are obtained as in Table
\ref{tab:results2}
\begin{table}[htbp]
    \caption{Similarity measure values under INR-environment}
    \label{tab:results2}
    \centering
    \begin{tabular}{ccc}
        \hline Similarity measure & Values & Ranking order\\
        \hline $(A^{*}, A_{i})_J$ &
        $
        \begin{array}{cc}
        (A^{*}, A_{1})_J=0.83489
\\
        (A^{*}, A_{2})_J=0.90254
\\
        (A^{*}, A_{3})_J=0.86578
\\
        (A^{*}, A_{4})_J=0.65791

        \end{array}
        $
        & $ A_{2}\succ  A_{3}\succ  A_{1}\succ  A_{4} $\\
        \hline $(A^{*}, A_{i})_D$&
        $
        \begin{array}{cc}
        (A^{*}, A_{1})_D=0.90283
\\
        (A^{*}, A_{2})_D=0.94872
\\
        (A^{*}, A_{3})_D=0.92618
\\
        (A^{*}, A_{4})_D=0.78961

        \end{array}
        $
        & $ A_{2}\succ  A_{3}\succ  A_{1}\succ  A_{4} $\\
        \hline $ (A^{*}, A_{i})_C$&
        $
        \begin{array}{cc}
        (A^{*}, A_{1})_C=0.90937
\\
        (A^{*}, A_{2})_C=0.95841
\\
        (A^{*}, A_{3})_C=0.96019
\\
        (A^{*}, A_{4})_C=0.80492

        \end{array}
        $
        & $ A_{3}\succ  A_{2}\succ  A_{1}\succ  A_{4} $\\
        \hline $ W(A^{*}, A_{i})_J$&
        $
        \begin{array}{cc}
        W(A^{*}, A_{1})_J=0.83534
\\
        W(A^{*}, A_{2})_J=0.75035
\\
        W(A^{*}, A_{3})_J=0.85113
\\
        W(A^{*}, A_{4})_J=0.66726

        \end{array}
        $
        & $ A_{3}\succ  A_{1}\succ  A_{2}\succ  A_{4} $\\
        \hline $ W(A^{*}, A_{i})_D$&
        $
        \begin{array}{cc}
        W(A^{*}, A_{1})_D=0.90259
\\
        W(A^{*}, A_{2})_D=0.94726
\\
        W(A^{*}, A_{3})_D=0.91794
\\
        W(A^{*}, A_{4})_D=0.79671

        \end{array}
        $
        & $ A_{2}\succ  A_{3}\succ  A_{1}\succ  A_{4} $\\
        \hline$ W(A^{*}, A_{i})_C$&
        $
        \begin{array}{cc}
        W(A^{*}, A_{1})_C=0.90911
\\
        W(A^{*}, A_{2})_C=0.95613
\\
        W(A^{*}, A_{3})_C=0.95695
\\
        W(A^{*}, A_{4})_C=0.81158

        \end{array}
        $
        & $ A_{3}\succ  A_{2}\succ  A_{1}\succ  A_{4} $\\
        \hline
    \end{tabular}
\end{table}
\item \textbf{Step 4:} Rankings of the alternatives are shown in last columnn of Table
\ref{tab:results2}.

\end{itemize}
\subsection{Multi-attribute decision making under INR-environment}
Let alternatives $A_i (i=1,2,...,k)$ are characterized by INR-values
for each criterion $C_j(i=1,2,...,r)$ as follows: \footnotesize
$$
A_i=\left\{\begin{array}{c}\langle
C_j,([{t^1}^L_{A_i}(C_j),{t^1}^U_{A_i}(C_j)],....,[{t^p}^L_{A_i}(C_j),{t^p}^U_{A_i}(C_j)]),
([{i^1}^L_{A_i}(C_j),{i^1}^U_{A_i}(C_j)],....,[{i^p}^L_{A_i}(C_j),{i^p}^U_{A_i}(C_j)]),\\
([{f^1}^L_{A_i}(C_j),{f^1}^U_{A_i}(C_j)],....,[{f^p}^L_{A_i}(C_j),{f^p}^U_{A_i}(C_j)])\rangle:
C_j\in C \end{array}\right\}.
$$
\normalsize
 For convenience,
$([{t^1}^L_{A_i}(C_j),{t^1}^U_{A_i}(C_j)],....,[{t^p}^L_{A_i}(C_j),{t^p}^U_{A_i}(C_j)]),
([{i^1}^L_{A_i}(C_j),{i^1}^U_{A_i}(C_j)],....,[{i^p}^L_{A_i}(C_j),$\\${i^p}^U_{A_i}(C_j)])$
and
$([{f^1}^L_{A_i}(C_j),{f^1}^U_{A_i}(C_j)],....,[{f^p}^L_{A_i}(C_j),{f^p}^U_{A_i}(C_j)])$
are denoted by
$([{t^1}^L_{ij}],....,[{t^p}^U_{ij}])$,\\$([{i^1}^L_{ij}],....,[{i^p}^U_{ij}])$
and $([{f^1}^L_{ij}],....,[{f^p}^U_{ij}])$, respectively. So
INR-value
$\theta_{ij}=\langle([{t^1}^L_{ij},{t^1}^U_{ij}],....,[{t^p}^L_{ij},{t^p}^U_{ij}]),$\\$([{i^1}^L_{ij},{i^1}^U_{ij}],....,[{i^p}^L_{ij},{i^p}^U_{ij}]),
([{f^1}^L_{ij},{f^1}^U_{ij}],....,[{f^p}^L_{ij},{f^p}^U_{ij}])\rangle
(i=1,2,...,k; j=1,2,...,r)$ which is generally obtained from the
evaluation of the alternative $A_i$ with related to the criteria
$C_j$ by opinion of expert or decision maker. Thus, INR-decision
matrix $D=[\theta_{ij}]_{k\times r}$ can be constructed.

In this study, for benefit criteria (BC) and cost criteria (CC)
ideal INR-values denoted by $A^*$ are defined as follows:
\footnotesize
\begin{itemize}
\item $\theta_j^*=\langle({t^1}^*_{j},....,{t^p}^*_{j}),({i^1}^*_{j},....,{i^p}^*_{j}),({f^1}^*_{j},....,{f^p}^*_{j})\rangle
=\left\langle\begin{array}{c}([max_i({t^1}^L_{ij}),max_i({t^1}^U_{ij})],....,[max_i({t^p}^L_{ij}),max_i({t^p}^U_{ij})]),\\
([min_i({i^1}^L_{ij}),min_i({i^1}^U_{ij})],....,[min_i({i^p}^L_{ij}),min_i({i^p}^U_{ij})]),
\\([min_i({f^1}^L_{ij}),min_i({f^1}^U_{ij})],....,[min_i({f^p}^L_{ij}),min_i({f^p}^U_{ij})])\end{array}\right\rangle$
\\
\\
\item $\theta_j^*=\langle({t^1}^*_{j},....,{t^p}^*_{j}),({i^1}^*_{j},....,{i^p}^*_{j}),({f^1}^*_{j},....,{f^p}^*_{j})\rangle
=\left\langle\begin{array}{c}([min_i({t^1}^L_{ij}),min_i({t^1}^U_{ij})],....,[min_i({t^p}^L_{ij}),min_i({t^p}^U_{ij})]),\\
([max_i({i^1}^L_{ij}),max_i({i^1}^U_{ij})],....,[max_i({i^p}^L_{ij}),max_i({i^p}^U_{ij})]),
\\([max_i({f^1}^L_{ij}),max_i({f^1}^U_{ij})],....,[max_i({f^p}^L_{ij}),max_i({f^p}^U_{ij})])\end{array}\right\rangle,$
\end{itemize}
\normalsize respectively.
\section*{Algorithm}
\begin{itemize}
\item \textbf{Step 1:} \textbf{Determination of BC and CC criteria.}

\item \textbf{Step 2:} \textbf{Determination of ideal  INR-values $A^*$ solution}

\item \textbf{Step 3:} \textbf{Calculation of weighted similarity measures}\\
In this step, using one of the Eq. (\ref{eq:wjac-inr}), Eq.
(\ref{eq:wdice-inr}) or Eq.(\ref{eq:wcos-inr}) weighted similarity
measures between the ideal alternative $A^*$ and INR-sets
$A_i(i=1,2,...,k)$ are calculated.

\item \textbf{Step 4:} \textbf{Ranking of the alternative}
According to the values obtained using one of the Eq.
(\ref{eq:wjac-inr}), Eq. (\ref{eq:wdice-inr}) or
Eq.(\ref{eq:wcos-inr}), the ranking order of all the alternatives
can be easily determined.
\end{itemize}

\section*{Illustrative example 2}\label{ex:inrsim}

In this example, alternatives and criteria given in previous
illustrative example  will be considered under INR-environment. The
four alternatives are to evaluated under the criteria by INR-values
provided by decision maker. These values are shown in INR-decision
matrix as follows: \footnotesize
$$D=[\gamma_{ij}]_{k\times r}=\left(
 \begin{array}{ccc}
    \left\langle \begin{array}{c}([.2,.3],[.2,.5],[.4,.7]),\\
    ([.3,.4],[.3,.6],[.5,.9]),\\([.2,.5],[.4,.7],[.8,.8])\\
    \end{array}\right\rangle  & \left\langle \begin{array}{c}
([.1,.5],[.4,.5],[.6,1]),\\([.2,.4],[.3,.7],[.7,.8]),\\([.1,.2],[.3,.8],[.3,.8])\\
\end{array}\right\rangle &
\left\langle \begin{array}{c}
([.0,.3],[.3,.5],[.3,9]),\\([.1,.2],[.5,.6],[.6,.6]),\\([.5,.5],[.6,.7],[.9,.9])\\
\end{array}\right\rangle \\
     \left\langle \begin{array}{c}([.1,.2],[.2,.8],[.4,.8]),\\([.4,.5],[.3,.6],[.5,.7]),\\([.1,.3],[.4,.5],[.8,.8])\\
     \end{array}\right\rangle  & \left\langle \begin{array}{c}
([.1,.4],[.4,.5],[.6,.6]),\\([.2,.3],[.3,.4],[.7,.8]),\\([.1,.5],[.3,.6],[.3,.7])\\
\end{array}\right\rangle & \left\langle \begin{array}{c}
([.0,.3],[.3,.4],[.3,5]),\\([.1,.6],[.5,.6],[.6,.7]),\\([.5,.8],[.6,.8],[.9,1])\\
\end{array}\right\rangle\\
    \left\langle \begin{array}{c}([.1,.4],[.2,.5],[.4,.6]),\\([.3,.4],[.3,.4],[.6,.7]),\\([.2,.3],[.4,.5],[.8,1])\\ \end{array}\right\rangle  & \left\langle \begin{array}{c}
([.2,.3],[.4,.5],[.6,.7]),\\([.2,.5],[.3,.6],[.7,.8]),\\([.1,.2],[.3,.4],[.4,.5])\\
\end{array}\right\rangle & \left\langle \begin{array}{c}
([.0,.1],[.3,.3],[.3,.4]),\\([.1,.2],[.5,.6],[.6,.7]),\\([.5,.6],[.6,.7],[.9,.9])\\
\end{array}\right\rangle \\
    \left\langle \begin{array}{c}([.1,.4],[.2,.4],[.4,.4]),\\([.3,.5],[.3,.6],[.5,.6]),\\([.2,.5],[.4,.6],[.8,.9])\\ \end{array}\right\rangle  & \left\langle \begin{array}{c}
([.1,.5],[.4,.5],[.4,.6]),\\
([.2,.4],[.3,.5],[.7,.9]),\\([.2,.2],[.3,.4],[.3,.4])\\
\end{array}\right\rangle & \left\langle \begin{array}{c}
([.0,.2],[.3,.4],[.3,.5]),\\
([.1,.4],[.5,.6],[.6,.8]),\\([.5,.6],[.6,.7],[.9,1])\\
\end{array}\right\rangle \\
  \end{array}
\right)$$ \normalsize
\begin{itemize}
\item \textbf{Step 1:} Let us consider $C_1$ and $C_2$ as  benefit criteria and  $C_3$ as cost
criterion.
\item \textbf{Step 2:} From INR-decision matrix, ideal alternative $A^*$ can be obtained
as follows:

\begin{eqnarray*}
A^*&=&\Big\{\langle([.2,.4],[.2,.8],[.4,.8]),([.3,.4],[.3,.4],[.5,.6]),([.1,.3],[.4,.5],[.8,.8])\rangle,\\
&&\langle([.2,.5],[.4,.5],[.6,1]),([.2,.3],[.3,.4],[.7,.8]),([.1,.2],[.3,.4],[.3,.4])\rangle\\
&&\langle([.0,.1],[.3,.3],[.3,.4]),([.1,.6],[.5,.6],[.6,.8]),([.5,.8],[.6,.8],[.9,1])\rangle\Big\}.
\end{eqnarray*}
\item \textbf{Step 3:} By using the Eq. (\ref{eq:wjac-inr}), Eq. (\ref{eq:wdice-inr}) and Eq. (\ref{eq:wcos-inr}) similarity
measures and weighted similarity measures are obtained as shown in
Table \ref{tab:results4}.

\item \textbf{Step 4:} Rankings of the alternatives are shown in last column of Table
\ref{tab:results4}.

\begin{table}[htbp]
    \caption{Similarity measure values and ranking of alternatives under INR-environment}
    \label{tab:results4}
    \centering
    \begin{tabular}{ccc}
        \hline Similarity measure & Values & Ranking order\\
        \hline $(A^{*}, A_{i})_J$ &
        $
        \begin{array}{cc}
        (A^{*}, A_{1})_J=0.83489
\\
        (A^{*}, A_{2})_J=0.95699
\\
        (A^{*}, A_{3})_J=0.95304
\\
        (A^{*}, A_{4})_J=0.94042
        \end{array}
        $
        & $ A_{2}\succ  A_{3}\succ  A_{4}\succ  A_{1} $\\
        \hline $(A^{*}, A_{i})_D$&
        $
        \begin{array}{cc}
        (A^{*}, A_{1})_D=0.90283
\\
        (A^{*}, A_{2})_D=0.97768
\\
        (A^{*}, A_{3})_D=0.97595
\\
        (A^{*}, A_{4})_D=0.96903
        \end{array}
        $
        & $ A_{2}\succ  A_{3}\succ  A_{4}\succ  A_{1} $\\
        \hline $ (A^{*}, A_{i})_C$&
        $
        \begin{array}{cc}
        (A^{*}, A_{1})_C=0.90937
\\
        (A^{*}, A_{2})_C=0.97790
\\
        (A^{*}, A_{3})_C=0.97758
\\
        (A^{*}, A_{4})_C=0.97007
        \end{array}
        $
        & $ A_{2}\succ  A_{3}\succ  A_{4}\succ  A_{1} $\\
        \hline $ W(A^{*}, A_{i})_J$&
        $
        \begin{array}{cc}
        W(A^{*}, A_{1})_J=0.83534
\\
        W(A^{*}, A_{2})_J=0.96355
\\
        W(A^{*}, A_{3})_J=0.95420
\\
        W(A^{*}, A_{4})_J=0.94270
        \end{array}
        $
        & $ A_{2}\succ  A_{3}\succ  A_{4}\succ  A_{1} $\\
        \hline $ W(A^{*}, A_{i})_D$&
        $
        \begin{array}{cc}
        W(A^{*}, A_{1})_D=0.90259
\\
        W(A^{*}, A_{2})_D=0.98114
\\
        W(A^{*}, A_{3})_D=0.97656
\\
        W(A^{*}, A_{4})_D=0,97021
        \end{array}
        $
        & $ A_{2}\succ  A_{3}\succ  A_{4}\succ  A_{1} $\\
        \hline$ W(A^{*}, A_{i})_C$&
        $
        \begin{array}{cc}
        W(A^{*}, A_{1})_C=0.90911
\\
        W(A^{*}, A_{2})_C=0.98138
\\
        W(A^{*}, A_{3})_C=0.97849
\\
        W(A^{*}, A_{4})_C=0.97112
        \end{array}
        $
        & $ A_{2}\succ  A_{3}\succ  A_{4}\succ  A_{1} $\\
        \hline
    \end{tabular}
\end{table}

\end{itemize}

\section{Consistency analysis of similarity measures based INR-sets}

In this section, to determine  which similarity measure gives more
consistent results, a method is given.

Let $A=\{A_1,A_2,...,A_n\}$ be  a set of alternatives,
$C=\{C_1,C_2,...,C_k\}$ be a set of criteria and $A^*$ be set of
ideal alternative values obtained from decision matrix defined in
illustrative example of similarity measures based on INR-set. Then,
consistency of the similarity measures based INR-values is define by
as follows:

$$
C( A^*, A_i)_{\Lambda}=\frac{1}{n}\sum_{i=1}^n|({
A^*}^L,{A_i}^L)_{\Lambda}-({A^*}^U,{A_i}^U)_{\Lambda}|.
$$
Here, ${A^*}^L$ and ${A^*}^U$ are determined with help of
INR-decision matrix using formula of benefit criteria (BC) and cost
criteria (CC) given  as follows: For $\Delta\in\{L=lower, U=upper\}$
\footnotesize
\begin{itemize}
\item ${\delta_j^*}^{\Delta}=\langle({{t^1}^*}^{\Delta}_{j},....,{{t^p}^*}^{\Delta}_{j}),
({{i^1}^*}^{\Delta}_{j},....,{{i^p}^*}^{\Delta}_{j}),({{f^1}^*}^{\Delta}_{j},....,{{f^p}^*}^{\Delta}_{j})\rangle
=\left\langle\begin{array}{c}(max_i({t^1}^{\Delta}_{ij}),....,max_i({t^p}^{\Delta}_{ij})),\\
(min_i({i^1}^{\Delta}_{ij}),....,min_i({i^p}^{\Delta}_{ij})),
\\(min_i({f^1}^{\Delta}_{ij}),....,min_i({f^p}^{\Delta}_{ij})])\end{array}\right\rangle$
\\
\item ${\delta_j^*}^{\Delta}=\langle({{t^1}^*}^{\Delta}_{j},....,{{t^p}^*}^{\Delta}_{j}),({{i^1}^*}^{\Delta}_{j},....,{{i^p}^*}^{\Delta}_{j}),({{f^1}^*}^{\Delta}_{j},....,{{f^p}^*}^{\Delta}_{j})\rangle
=\left\langle\begin{array}{c}(min_i({t^1}^{\Delta}_{ij}),....,min_i({t^p}^{\Delta}_{ij})),\\
(max_i({i^1}^{\Delta}_{ij}),....,max_i({i^p}^{\Delta}_{ij})),
\\(max_i({f^1}^{\Delta}_{ij}),....,max_i({f^p}^{\Delta}_{ij}))\end{array}\right\rangle,$
\end{itemize}
\normalsize
respectively.\\
 Also ${A_i}^L=\Big\{\Big\langle
({t^1}^{L}_{A_i}(C_j),...,{t^p}^{L}_{A_i}(C_j)),
({i^1}^{L}_{A_i}(C_j),...,{i^p}^{L}_{A_i}(C_j)),
({f^1}^{L}_{A_i}(C_j),...,{f^p}^{L}_{A_i}(C_j))\Big\rangle: C_j\in C
\textrm{ and } i\in I_p\Big\}$ and ${A_i}^U=\Big\{\Big\langle
({t^1}^{U}_{A_i}(C_j),...,{t^p}^{U}_{A_i}(C_j)),
({i^1}^{U}_{A_i}(C_j),...,{i^p}^{U}_{A_i}(C_j)),
({f^1}^{U}_{A_i}(C_j),...,{f^p}^{U}_{A_i}(C_j))\Big\rangle: C_j\in C
\textrm{ and } i\in I_p\Big\}$.

\begin{exmp} Let us consider the Example \ref{ex:inrsim}. Then
for all $\Lambda\in \{J,D,C\}$, results and orderings are obtained
as in Table \ref{tab:results5},
\begin{table}[htbp]
    \caption{Consistency degrees of similarity measures under INR environment}
    \label{tab:results5}
    \centering
    \begin{tabular}{cccc}
        \hline Similarity measure & $({A^*}^L, A^L_{i})$ & $({A^*}^U, A^U_{i})$& $C(A^*,A_i)_{\Lambda}$\\
        \hline Jaccard &
        $
        \begin{array}{cc}
        0,99220
\\
        0,99313
\\
        0,99118
\\
        0,98014
        \end{array}
        $
        &
        $
        \begin{array}{cc}
0,87306
\\
0,93997
\\
0,93188
\\
0,92097
        \end{array}
        $
        & $0,07269$

\\
        \hline Dice &
          $
        \begin{array}{cc}
        0,99608\\
        0,99655\\
        0,99556\\
        0,98987
        \end{array}
        $&
        $
        \begin{array}{cc}
        0,93187\\
        0,96832\\
        0,96472\\
        0,95835
        \end{array}
        $
        & $0,03870$
\\
        \hline Cosine &
        $
        \begin{array}{cc}
0,99649\\
0,99660\\
0,99587\\
0,99089
        \end{array}
        $
        &
        $
        \begin{array}{cc}
        0,94012\\
0,96869\\
0,96947\\
0,95975
        \end{array}
        $
        & $0,03545$
\\
        \hline $ Weighted \, Jaccard$&
        $
         \begin{array}{cc}
0,99207
\\
0,99353
\\
0,99145
\\
0,98376
        \end{array}
        $
        &
        $
        \begin{array}{cc}
        0,86958
\\
0,94931
\\
0,93316
\\
0,92248
        \end{array}
        $
        & $ 0,07157
 $\\
        \hline $ Weighted \, Dice$&
        $
         \begin{array}{cc}
0,99602
\\
0,99674
\\
0,99569
\\
0,99173
        \end{array}
        $
        &
        $
        \begin{array}{cc}
        0,92984
\\
0,97337
\\
0,96541
\\
0,95912
        \end{array}
        $
        & $0,03811
$\\
        \hline$Weighted \, Cosine$&
        $
         \begin{array}{cc}
0,99649
\\
0,99679
\\
0,99599
\\
0,99249
        \end{array}
        $
        &
        $
        \begin{array}{cc}
        0,93678
\\
0,97377
\\
0,97097
\\
0,96049
        \end{array}
        $
        & $ 0,03494
 $\\
        \hline
    \end{tabular}
\end{table}
\end{exmp}
Note that, $C(A^*,A_i)_{J}\geq C(A^*,A_i)_{D}\geq C(A^*,A_i)_{C}$.
Since consistency degree of Jaccard similarity measure under INR
environment  is higher than consistency degrees of Dice and Cosine
similarity measures, it is more convenient using the Jaccard
similarity measure for discussed problem.

\section{Conclusion}
In this paper,  for SVNR-set and INR-sets three similarity measures
method developed based on Jaccard, Dice and Cosine similarity
measures. Furthermore, applications of proposed similarity measure
methods are given in multi-criteria decision making and a method is
developed to compare similarity measures of INR-sets, and an
application of this method is given.  However, I hope that the main
thrust of proposed formulas will be in the field of equipment
evaluation, data mining and investment decision making. Also in
future, similarity measure methods for INR-sets can be proposed
based on the methods other than Jaccard, Dice and Cosine similarity
measures.


\end{document}